\documentclass[11pt,reqno]{amsart}
\usepackage{microtype}

\usepackage[shortlabels]{enumitem}
\setlist[enumerate]{label={(\arabic*)}}

\usepackage{amssymb}
\usepackage[dvipsnames]{xcolor}
\usepackage[colorlinks=true,linkcolor=Maroon,citecolor=OliveGreen,backref]{hyperref}
\usepackage[abbrev, alphabetic]{amsrefs}  

\usepackage[capitalize]{cleveref} 
\crefname{equation}{}{}

\numberwithin{equation}{section}

\newtheorem{lemma}{Lemma}[section]
\newtheorem{theorem}[lemma]{Theorem}
\newtheorem{proposition}[lemma]{Proposition}
\newtheorem{corollary}[lemma]{Corollary}

\theoremstyle{definition}
\newtheorem{remark}[lemma]{Remark}

\newtheorem{problem}[lemma]{Problem}

\newcommand\opr[1]{\operatorname{#1}}

\newcommand{\eps}{\epsilon}

\def\Q{\mathbf{Q}}
\def\C{\mathbf{C}}
\def\Z{\mathbf{Z}}

\def\cC{\mathcal{C}}

\newcommand\br[1]{{\left(#1\right)}}

\newcommand\floor[1]{\left\lfloor{#1}\right\rfloor}

\newcommand{\Zmod}[1]{\Z/{#1}\Z}

\def\GL{\opr{GL}}

\def\AGL{\opr{AGL}}

\def\sm{\smallsetminus}

\def\cube{\opr{cube}}

\usepackage{todonotes}

\begin{document}

\title[Normal covering numbers for $S_n$ and $A_n$]{Normal covering numbers for $S_n$ and $A_n$ and additive combinatorics}

\author{Sean Eberhard}
\address{Sean Eberhard, Mathematics Institute, Zeeman Building, University of Warwick, Coventry~CV4~7AL, UK}
\email{sean.eberhard@warwick.ac.uk}

\author{Connor Mellon}
\address{Connor Mellon, Mathematical Sciences Research Centre, Queen's University Belfast, Belfast BT7~1NN, UK}
\email{cmellon07@qub.ac.uk}

\thanks{SE is supported by the Royal Society. CM was supported by an MSRC summer internship.}

\begin{abstract}
    The normal covering number $\gamma(G)$ of a noncyclic group $G$ is the minimum number of proper subgroups whose conjugates cover the group. We give various estimates for $\gamma(S_n)$ and $\gamma(A_n)$ depending on the arithmetic structure of $n$.
    In particular we determine the limsups over $\gamma(S_n) / n$ and $\gamma(A_n) / n$ over the sequences of even and odd integers, as well as the liminf of $\gamma(S_n) / n$ over even integers.
    In general we explain how the values of $\gamma(S_n) / n$ and $\gamma(A_n) / n$ are related to problems in additive combinatorics.
    These results answer most of the questions posed by Bubboloni, Praeger, and Spiga as Problem 20.17 of the Kourovka Notebook.
\end{abstract}

\maketitle

\section{Introduction}

A \emph{normal covering} of a group $G$ is a family $\cC$ of proper subgroups $H < G$ whose conjugates together cover the group: $G = \bigcup_{g \in G} \bigcup_{H \in \cC} H^g$.
Clearly a group has a normal covering if and only if the group is noncyclic.
The \emph{normal covering number} $\gamma(G)$ of a noncyclic group $G$ is the minimal cardinality of a normal covering of $G$.
In this paper we only consider normal coverings and so, to be brief, we will often refer to a normal covering $\cC$ simply as a \emph{covering} or a \emph{covering family}.
Note that if $G$ if finite (or even finitely generated) there is no loss in generality in assuming $\cC$ contains only maximal subgroups. 

There are plenty of groups for which $\gamma(G) = 1$.
For example, $\GL_n(\C)$ ($n \ge 2$) is covered by the conjugates of the Borel subgroup consisting of upper triangular matrices.
However, a well-known theorem of Jordan asserts that no finite group $G$ is the union of conjugates of a proper subgroup.
Thus $\gamma(G) > 1$ for every finite (noncyclic) group $G$.
There is a wide variety of finite groups with $\gamma(G) = 2$ (see Garonzi--Lucchini~\cite{garonzi--lucchini}).
The finite simple groups with $\gamma(G) = 2$ were recently classified by Bubboloni, Spiga, and Weigel~\cite{BSW}.

Normal covering numbers arise naturally in the study of intersective polynomials. A polynomial $f \in \Z[x]$ is \emph{intersective} if it has no root in $\Q$ but it has a root modulo $m$ for every positive integer $m$.
Examples include $f(x) = (x^2 - 13)(x^2 - 17)(x^2 - 221)$ and $f(x) = (x^3 - 19) (x^2 + x + 1)$.
Let $n = \deg(f)$ and let $G$ be the Galois group of $f$, which we may consider to be a subgroup of $S_n$. Suppose $f$ factorizes as $f_1 \cdots f_m$. Let $K$ be a splitting field of $f$ and let $\theta_i \in K$ be a root of $f_i$ for each $i = 1, \dots, m$. Then it follows from the existence of the Frobenius element and the Chebotarev density theorem that the subgroups $H_i = \opr{Gal}(K / \Q(\theta_i))$ for $1 \le i \le m$ form a normal covering of $G$.
Thus $m \ge \gamma(G)$.
In particular, the theorem of Jordan quoted above implies that $f$ must be reducible.
The two examples given above have Galois groups $C_2 \times C_2$ and $S_3$, which have $\gamma(G) = 3$ and $\gamma(G) = 2$, respectively.
See \cite{berend--bilu} for more on this.

Since a generic Galois extension of $\Q$ has Galois group $S_n$ or $A_n$, there is particular reason to be interested in the values of $\gamma(S_n)$ and $\gamma(A_n)$.
For exact values for small $n$, see \cite{BPS2014}*{Table~1}.
These numbers were first investigated in earnest by Bubboloni and Praeger in \cite{BP2011}, who proved bounds of the form $a \phi(n) \le \gamma(G) \le b n$ for constants $a, b > 0$, where $\phi$ is Euler's totient function.
A linear (but weak) lower bound was subsequently established by Bubboloni, Praeger, and Spiga~\cite{BPS2013}.
An exact formula for $\gamma(S_n)$ was conjectured in \cite{BPS2014} (and \cite{kourovka}*{18.23}),
and then refuted in \cite{BPS2019}, using a construction of Mar\'oti.
The problem of determining the limsups and liminfs of $\gamma(S_n) / n$ and $\gamma(A_n) / n$ over even and odd integers, respectively ($8$ questions in all) is now listed as Problem 20.17 of the Kourovka Notebook~\cite{kourovka}.

We add a new chapter to this story by showing how normal covering numbers for $S_n$ and $A_n$ are related to some elementary puzzles in additive combinatorics. Using this translation, we obtain some new exact formulas for some special cases. For example, we have $\gamma(A_p) = \floor{(p+1)/3}$ for all sufficiently large primes $p$ not of the form $(q^d-1)/(q-1)$.
Moreover, we find $5$ of the $8$ limits in question ($3$ of which are new),
and obtain much tighter bounds on the remaining $3$ than were previously known.

\begin{theorem}
    \label{thm:limits}
    We have the following asymptotic estimates for the normal covering numbers $\gamma(S_n)$ and $\gamma(A_n)$:
    \begin{align*}
        &\limsup_{n~\textup{even}} \gamma(S_n) / n = 1/4,
        &&\limsup_{n~\textup{even}} \gamma(A_n) / n = 1/4, \\
        &\limsup_{n~\textup{odd}} \gamma(S_n) / n = 1/2,
        &&\limsup_{n~\textup{odd}} \gamma(A_n) / n = 1/3, \\
        &\liminf_{n~\textup{even}} \gamma(S_n) / n = 1/6,
        &&\liminf_{n~\textup{even}} \gamma(A_n) / n \in [1/18, 1/6], \\
        &\liminf_{n~\textup{odd}} \gamma(S_n) / n \in [1/12, 1/6],
        &&\liminf_{n~\textup{odd}} \gamma(A_n) / n \in [1/6, 4/15]. \\
    \end{align*}
\end{theorem}

Moreover, the three outstanding questions are essentially reduced to two questions in additive combinatorics related to sum-free and cube-free sets with coprimality conditions.
Let us now describe these problems. First, call a subset $X \subset \Zmod n$ \emph{coprime-sum-free} if there does not exist $x, y, x+y \in X$ with $\gcd(x, y, n) = 1$.

\begin{problem}
    \label{problem-1}
    Determine the size of the largest symmetric coprime-sum-free subset of $\Zmod n$ for all $n \ge 1$.
\end{problem}

Absent the coprimality condition, the solution to this problem is a well-known consequence of Kneser's generalization of the Cauchy--Davenport Theorem: see for example Diananda--Yap~\cite{DY}. The answer is as follows. If $n$ has a prime factor $p \equiv 2 \pmod 3$ then the answer is exactly $(1+1/p)n/3$, where $p$ is the smallest such factor.
Otherwise, if $n$ is divisible by $3$ then the answer is $n/3$.
Finally, if every prime factor of $n$ is congruent to $1 \pmod 3$ then the answer is $(1-1/n)/3$.

The coprimality condition makes the problem more complicated.
If $n$ is divisible by $6$, the example $X = \{x \in \Zmod n : 2 \mid x ~\text{or} ~ 3 \mid x\}$ shows that $X$ can have size $(2/3)n$, and moreover we can show (see \Cref{prop:coprime-sum-free-2/3}) that no coprime-sum-free set can have size greater than $(2/3)n$. This fact turns out to give the value $\liminf_{n~\textup{even}} \gamma(S_n) / n = 1/6$.
But what if $n$ is odd? Then the constant ought to drop to $7/15$, and even smaller if $n$ is not divisible by $3$, etc.
A complete answer to \Cref{problem-1} would clarify how $\gamma(S_n) / n$ ($n$~even) and $\gamma(A_n) / n$ ($n$~odd) depend on the presence of small factors.

Next, for integers $x, y, z$ define
\[
    \cube(x, y, z) = \{x, y, z, x+y, x+z, y+z, x+y+z\}.
\]
Call a subset $X \subset \{1, \dots, n-1\}$ \emph{coprime-cube-free} if there does not exist $x, y, z$ with $\gcd(x, y, z, n) = 1$ and $\cube(x, y, z) \subset X$.

\begin{problem}
    \label{problem-2}
    Determine the size of the largest symmetric coprime-cube-free subset of $\{1, \dots, n-1\}$ for all $n \ge 1$.
\end{problem}

Even without the coprimality condition, avoidance problems for higher-order configurations like cubes can be difficult. For example, recently Meng~\cite{meng} considered a version of the above problem without the coprimality condition.

If $n$ is divisible by $3$, the example $X = \{x : 3 \nmid x\}$ shows that $X$ can have size $(2/3)n$.
We show (see \Cref{prop:coprime-cube-free}) that $|X| \le (5/6) n + O(1)$.
A solution to this problem should lead to the determination of the limits $\liminf_{n~\textup{even}} \gamma(A_n) / n$ and $\liminf_{n~\textup{odd}} \gamma(S_n) / n$.
Again, it is reasonable to expect that the solution should be sensitive to the presence of small prime factors, for example the prime factor $3$.
Consequently it is reasonable to expect for example that
\[
    \liminf_{\gcd(n, 6) = 1} \gamma(S_n) / n > \liminf_{n~\text{odd}} \gamma(S_n) / n.
\]

Throughout the paper, the cases ($S_n$, $n$ even) and ($A_n$, $n$ odd) will often be considered together, as will the complementary cases ($S_n$, $n$ odd) and ($A_n$, $n$ even).
The reason for this is easy to explain.
First, since we may use $A_n$ in a covering family of $S_n$, covering $S_n$ is almost the same as covering all odd permutations, while obviously covering $A_n$ is the same as covering all even permutations.
It turns out that the most difficult permutations to cover are the ones with only $3$ or $4$ cycles.
Now recall that if $\pi \in S_n$ has $k$ cycles then the sign of $\pi$ is $(-1)^{n-k}$.
Therefore $k = 3$ corresponds to the cases ($S_n$, $n$ even) and ($A_n$, $n$ odd)
while $k = 4$ corresponds to the cases ($S_n$, $n$ odd) and ($A_n$, $n$ even).
The case $k = 4$ is more difficult.

\section{Upper bounds}

In this section we summarize the known upper bounds and add two new observations (\Cref{prop:new-upper-1,prop:new-upper-2}).
Let us emphasize that the results with citations are not new, nor are the proofs, but it is valuable to recall them here to give context to the arguments for the lower bounds that we will give in the next sections.

\begin{proposition}[\cite{BP2011}*{Theorem~A}]
    \label{prop:basic-upper-bounds}
    For $G = S_n$ or $A_n$ with $n \ge 4$,
    \[
        \gamma(G) \le \begin{cases}
            \floor{n/2} &\text{if}~ G = S_n,~n~\text{odd},\\
            \floor{n/3} + 1 &\text{if}~ G = A_n,~n~\text{odd},\\
            \floor{n/4} + 1 &\text{if $n$ is even}.
        \end{cases}
    \]
\end{proposition}

For the sake of comparison with our lower bound arguments, we indicate the covering families and sketch the proof.

\begin{proof}[Sketch]
    \emph{Case 1: $G = S_n$ with $n$ odd.}
    Let $p$ be the smallest prime factor of $n$. If $p < n$ then
    \[
        \{S_k \times S_{n-k} : 1 \le k \le (n-1) / 2, ~ p \nmid k\} \cup \{S_p \wr S_{n/p}\}
    \]
    is a covering family of cardinality at most $(n-1)/2$. If $p = n$ then consider instead
    \[
        \{S_k \times S_{n-k}: 1 < k \le (n-1)/2\} \cup \{ \AGL_1(p) \}.
    \]
    The subgroup $\AGL_1(p)$ covers the cycle types $(n)$ and $(n-1, 1)$, and the subgroups $S_k \times S_{n-k}$ with $1 < k \le (n-1) / 2$ cover all others.

    \emph{Case 2: $G = A_n$ with $n$ odd.}
    Since $n$ is odd, every element of $A_n$ has an odd number of cycles, and is therefore either an $n$-cycle or has a cycle of length at most $n/3$. A covering family is therefore given by
    \[
        \{(S_k \times S_{n-k}) \cap A_n : k \le n/3\} \cup \{ M \cap A_n \},
    \]
    where $M$ is any maximal subgroup of $S_n$ other than $A_n$ containing an $n$-cycle.
    Indeed, $|M : M \cap A_n| = 2$, so $M \cap A_n$ is normalized by at least one odd permutation and therefore $M \cap A_n$ intersects both $A_n$-conjugacy classes of $n$-cycles.
    If $n$ is composite one may take $M = S_p \wr S_{n/p}$ where $p$ is the smallest prime factor of $n$, and if $n$ is prime one may take $M = \AGL_1(p)$.

    \emph{Case 3: $G = S_n$ or $A_n$ with $n$ even.}
    Every permutation without odd cycles is contained in a conjugate of $S_2 \wr S_{n/2}$, and so are the permutations of type $(n/2, n/2)$, and every other permutation is contained in $S_k \times S_{n-k}$ for some odd $k < n/2$. Therefore
    \[
        \{(S_k \times S_{n-k}) \cap G : 1 \le k < n/2, ~k~\text{odd}\} \cup \{ (S_2 \wr S_{n/2}) \cap G\}
    \]
    is a covering family of size $\floor{n/4} + 1$.
\end{proof}
\begin{remark}
    \label{remark:known-exact-things}
    Equality holds above for $S_p$ ($p$ prime), $A_{2^k}$ ($k \ne 3$), and $A_{2p}$ ($p$ prime):
    see \cite{BP2011}*{Propositions 7.1, 7.5(c), and 7.6(c)}.
    We will show later that equality also holds for $S_{2^k}$ and $S_{2p}$ ($p$ prime) and for $A_p$ for almost all primes $p \equiv 2 \pmod 3$.
\end{remark}

The following bound improves the bound in \Cref{prop:basic-upper-bounds} by $1$ in one special case.
The covering family may seem mysterious for now.
Its origin will be clarified later when the connection to additive combinatorics is fleshed out.

\begin{proposition}
    \label{prop:new-upper-1}
    Suppose $p \ge 5$ is prime and $p \equiv 1 \pmod 3$. Then
    \[
        \gamma(A_p) \le (p-1) / 3.
    \]
\end{proposition}
\begin{proof}
    Write $p = 3m+1$, so $m = (p-1)/3$.
    Let $K$ be the set of integers $k \in \{1, \dots, p-1\}$ such that $(m+1)k \in \{1, \dots, m-1\} \pmod p$.
    Since $p$ is prime, $|K| = m-1$.
    Let $I$ be the middle interval $\{m, \dots, 2m+1\}$ mod $p$ and note
    \[
        \Zmod p = \{0\} \sqcup (m+1)K \sqcup I \sqcup -(m+1)K.
    \]
    We claim that
    \[
        \cC = \{(S_k \times S_{p-k}) \cap A_p : k \in K\} \cup \{\AGL_1(p) \cap A_p\}
    \]
    is a covering family of cardinality $|\cC| = |K| + 1 = m = (p-1)/3$.
    Let $\pi \in A_p$ and suppose $\pi$ has cycle type $(x_1, \dots, x_\ell)$. Note that $\ell$ is odd.

    \emph{Case $\ell = 1$:} If $\ell = 1$ then $\pi$ is a $p$-cycle, so $\pi$ is covered by $\AGL_1(p) \cap A_p$.

    \emph{Case $\ell = 3$:} If $x_i \in K$ or $p - x_i \in K$ for some $i$ then $\pi$ is covered by $S_{x_i} \times S_{p-x_i}$, so assume $x_1, x_2, x_3 \notin K \cup (p-K)$. For $y_i = (m+1)x_i \bmod p$ for $i = 1,2,3$. Then $y_1, y_2, y_3 \in I$ and $y_1 + y_2 + y_3 = (m+1)p = 0$. This implies that $y_1, y_2, y_3 \in \{\pm m, \pm(m+1)\}$.
    Noting that $(m+1) (p-1)/2 = 3(m+1)m/2 \equiv m \pmod p$, it follows that $x_1, x_2, x_3 \in \{\pm 1, \pm (p-1)/2\} \pmod p$,
    so $\pi$ must have type $(1, (p-1)/2, (p-1)/2)$.
    Thus $\pi$ is covered by $\AGL_1(p) \cap A_p$.

    \emph{Case $\ell \ge 5$:} We are obviously done if $x_i = 1$ for all $i$, so assume without loss of generality that $x_1 > 1$. Let $x = x_1$, $y = x_2 + x_3$, $z = x_4 + \cdots + x_\ell$. Arguing as in the previous case, we are done unless $x, y, z \notin K \cup (p-K)$, and if $x, y, z \in K \cup (p-K)$ then since $x + y + z = p$ it follows that $(x, y, z)$ is $(1, (p-1)/2, (p-1)/2)$ in some order, but this is impossible because $x, y, z > 1$.
\end{proof}

If $n$ has small prime factors then the values of $\gamma(S_n)$ and $\gamma(A_n)$ are much smaller than the upper bound in \Cref{prop:basic-upper-bounds} indicated above.

\begin{proposition}[\cite{BPS2014}*{Propostion~3.1}]
    \label{prop:BPS-upper-bound}
    Assume $G = S_n$ and $n$ is even or $G = A_n$ and $n$ is odd.
    Suppose $n$ is not a prime power and let $p_1 < p_2$ be the smallest prime factors of $n$.
    Then
    \[
        \gamma(G) \le \frac{n}{2} \br{1 - \frac1{p_1}} \br{1 - \frac1{p_2}} + 2.
    \]
\end{proposition}
\begin{proof}
    A covering family is given by $\cC_1 \cup \cC_2$, where
    \begin{align*}
        & \cC_1 = \{(S_k \times S_{n-k}) \cap G : 1 \le k < n/2,\ \gcd(k, p_1p_2) = 1\} \\
        & \cC_2 = \{(S_{p_1} \wr S_{n/p_1}) \cap G, (S_{p_2} \wr S_{n/p_2}) \cap G\}.
    \end{align*}
    Let $\pi \in G$. For $i=1,2$, if all the cycles of $\pi$ have length divisible by $p_i$ then $\pi$ is contained in a conjugate of $(S_{p_i} \wr S_{n/p_i}) \cap G \in \cC_2$. Otherwise $\pi$ must have an invariant set of cardinality coprime with $p_1p_2$, given by either a cycle of $\pi$ or a sum of two cycles, so $\pi$ is contained in a conjugate of some $(S_k \times S_{n-k}) \cap G \in \cC_1$.
\end{proof}

The following new observation is partly inspired by Mar\'oti's construction described in \cite{BPS2014}*{Section~2}.
Recall that $\phi(n)$ denotes Euler's totient function and $\omega(n)$ denotes the number of prime factors of $n$.

\begin{proposition}
    \label{prop:new-upper-2}
    Assume $G = S_n$ and $n$ is odd or $G = A_n$ and $n$ is even.
    Suppose $n$ is divisible by $3$. Then
    \[
        \gamma(G) \le \begin{cases}
            n/6 + \phi(n)/2 + \omega(n)&\text{if}~ G=A_n, ~n~ \text{even},\\
            n/6 + \phi(n)/2 + \omega(n) + 1&\text{if}~ G=S_n, ~n~ \text{odd}.\\
            
        \end{cases}
    \]
\end{proposition}
\begin{proof}
    A covering family is given by $\cC_1 \cup \cC_2 \cup \cC_3$ where
    \begin{align*}
        & \cC_1 = \{(S_k \times S_{n-k}) \cap G : 1 \le k \le n/2,\ 3 \mid k\}, \\
        & \cC_2 = \{(S_k \times S_{n-k}) \cap G : 1 \le k \le n/2,\ \gcd(n, k) = 1\}. \\
        & \cC_3 = \{(S_p \wr S_{n/p}) \cap G : p \mid n\},
    \end{align*}
    together with $\cC_4 = \{A_n\}$ if $G = S_n$.
    Note $|\cC_1| \le n/6$, $|\cC_2| = \phi(n)/2$, $|\cC_3| = \omega(n)$, $|\cC_4| \le 1$.
    To see why this works, let $\pi \in G$, and if $G = S_n$ assume $\pi \notin A_n$.
    Then $\pi$ has an even number of cycles.
    First suppose $\pi$ has only two cycles, say of lengths $k$ and $n-k$ where $k \le n/2$. If $\gcd(k, n) > 1$ then $\pi$ is contained in a conjugate of $(S_p \wr S_{n/p}) \cap G \in \cC_3$ where $p$ is any prime dividing $\gcd(k, n)$.
    Otherwise $\pi$ is contained in a conjugate of $(S_k \times S_{n-k}) \cap G \in \cC_2$.
    Now suppose $\pi$ has at least four cycles. Let $a, b, c$ be the lengths of three of the cycles. Then some number $k \in \{a, b, c, a+b, a+c, b+c, a+b+c\}$ is divisible by $3$, and $\pi$ is contained in a conjugate of $(S_k \times S_{n-k}) \cap G \in \cC_1$.
\end{proof}

\section{Cycle types covered by transitive subgroups}

To reduce to additive combinatorics we need to argue that transitive subgroups cannot help much.
For this we need to know when elements with a small number of cycles are covered by a transitive subgroup.
This is the content of the next two lemmas.

\begin{lemma}
    \label{lem:imprimitive}
    Let $\pi \in S_n$ have cycle type $(x_1, \dots, x_k)$, where $k > 1$.
    \begin{enumerate}
        \item If $\gcd(x_1, \dots, x_k) > 1$, then $\pi$ is contained in an imprimitive transitive subgroup conjugate to $S_p \wr S_{n/p}$ where $p$ is any prime divisor of $\gcd(x_1, \dots, x_k)$.
        \item If $\gcd(x_1, \dots, x_k) = 1$ then $\pi$ is contained in an imprimitive transitive subgroup if and only if there is a partition $\Pi$ of $\{1, \dots, k\}$, neither discrete nor trivial, such that if $b = \gcd(\sum_{i \in C} x_i : C \in \Pi)$ then for each cell $C \in \Pi$ and $i \in C$ we have $(\sum_{j \in C} x_j) / b \mid x_i$.
        In this case $\pi$ is contained in a conjugate of $S_b \wr S_{n/b}$.
    \end{enumerate}

    In particular, if $2 \le k \le 4$ and $\gcd(x_1, \dots, x_k) = 1$ then $\pi$ is contained in an imprimitive transitive subgroup if and only if either $k=3$ and for some permutation of $x_1, x_2, x_3$ we have
    \[
        (x_1 + x_2) / \gcd(x_3, n) \mid x_1, x_2,
    \]
    or $k = 4$ and for some permutation of $x_1, x_2, x_3, x_4$ we have either
    \begin{enumerate}[(a)]
        \item $(x_1 + x_2 + x_3) / \gcd(x_4, n) \mid x_1, x_2, x_3$,
        \item $(x_1 + x_2) / \gcd(x_3, x_4, n) \mid x_1, x_2$, or
        \item $(x_1 + x_2) / \gcd(x_1 + x_2, n) \mid x_1, x_2$ and $(x_3 + x_4) / \gcd(x_1 + x_2, n) \mid x_3, x_4$.
    \end{enumerate}
\end{lemma}
\begin{proof}(Cf.~\cite{BPS2014}*{Section~4}.)
    It is well-known and easy to check that $\pi$ is contained in an imprimitive transitive subgroup if and only if there is a nontrivial proper divisor $b \mid n$, a partition $(d_1, \dots, d_m)$ of $n/b$, and a partition of $\Omega$ into $\pi$-invariant sets $A_1, \dots, A_m$ ($m \le k$) such that $|A_i| = d_i b$ and every cycle of $\pi \mid A_i$ is divisible by $d_i$. (See, e.g., \cite{BPS2013}*{Section~4} or \cite{EFK}*{p.~4}.)
    The integer $b$ is the block size in the invariant partition and $(d_1, \dots, d_m)$ is the cycle type of the permutation $\bar \pi$ induced on the set of blocks.

    If $m = 1$ then we must find a nontrivial proper divisor $b$ of $n$ such that every cycle of $\pi$ is divisible by $d_1 = n / b$.
    Similarly, if $m = k$, then we must find a nontrivial proper divisor $b$ of $n$ such that every cycle of $\pi$ is divisible by $b$.
    In either case it is possible if and only if $\gcd(x_1, \dots, x_k) > 1$. This gives condition (1).

    Now suppose $1 < m < k$.
    Then we must collect the cycles of $\pi$ into sets $A_1, \dots, A_m \subseteq \{1,\dots,n\}$, and this defines a partition $\Pi$ of $\{1, \dots, k\}$ that is neither discrete nor trivial. Moreover we must find a nontrivial common divisor $b$ of $|A_1|, \dots, |A_m|$ such that every cycle of $\pi$ in $A_i$ is divisible by $d_i = |A_i| / b$. It suffices to test
    \[
        b = \gcd(|A_1|, \dots, |A_m|) = \gcd\br{\sum_{i \in C} x_i : C \in \Pi}.
    \]
    The condition that every cycle of $\pi$ in $A_i$ is divisible by $d_i$ is equivalent to requiring that $(\sum_{j \in C} x_j) / b \mid x_i$ for each cell $C \in \Pi$ and element $i \in C$.
    Moreover, if this condition holds, then $b$ is guaranteed to be a nontrivial proper divisor of $n$.
    Indeed, since $\Pi$ is neither discrete nor trivial, there is some cell $C$ such that $1 < |C| < k$. 
    For this cell, the condition that $(\sum_{j \in C} x_j) / b \mid x_i$ implies $b > 1$.
    On the other hand, we also have $b \le \sum_{j \in C} x_j < n$.
    This gives condition (2).
\end{proof}

The following lemma extends \cite{BPS2019}*{Lemma~7.1}, and follows \cite{GMPS2016}*{Theorem~1.1} and \cite{GMPS2015}*{Theorem~1.5} and a careful inspection of the tables. The restriction $n > 36$ enables us to keep the list reasonably short, but it can be omitted at the expense of including a large number of exceptional cases.

\begin{lemma}
    \label{lem:primitive}
    Assume $n > 36$.
    Let $\pi \in S_n$ have cycle type $\tau = (x_1, \dots, x_k)$,
    where $2 \le k \le 4$ and $\gcd(x_1, \dots, x_k) = 1$.
    If $\pi$ is contained in a primitive transitive subgroup smaller than $A_n$ then one of the following holds (throughout $q$ denotes a prime power and $p$ denotes a prime):\footnote{%
        We have excluded the case listed as line 3 of \cite{BPS2019}*{Table~1}, which we suspect was copied from line 4 of \cite{GMPS2015}*{Table~5} ignoring the restriction $i \mid p^{d-1} -1$.
    }

    Case $k=2$:
    \begin{enumerate}
        \item $n = q$, $\tau = (1, q - 1)$,
        \item $n = q+1$, $\tau = (1, q)$,
    \end{enumerate}

    Case $k=3$:
    \begin{enumerate}[resume]
        \item $n = q$, $q$ odd, $\tau = (1, (q-1)/2, (q-1)/2)$,
        \item $n = p^2$, $p$ odd, $\tau = (1, p-1, p(p-1))$,
        \item $n = (q^d-1) / (q-1)$, $d = d_1 + d_2$, $\gcd(d_1, d_2) = 1$,
        \[
            \tau = \br{\frac{q^{d_1} - 1}{q-1}, \frac{q^{d_2} - 1}{q - 1}, \frac{(q^{d_1} - 1)(q^{d_2} - 1)}{q-1}},
        \]
    \end{enumerate}

    Case $k=4$:
    \begin{enumerate}[resume]
        \item $n = q \equiv 1 \pmod 3$, $\tau = (1, (q-1) / 3, (q-1) / 3, (q-1) / 3)$,
        \item $n = 2^d$, $d = d_1+d_2$, $\tau = (1, 2^{d_1}-1, 2^{d_2}-1, (2^{d_1} -1)(2^{d_2} -1))$,
        \item $n = 2^d$, $\tau = (1, 1, 2^{d-1}-1, 2^{d-1}-1)$,
        \item $n = m^2$, $\tau = (k_1 k_2, k_1 (m-k_2), (m-k_1) k_2, (m-k_1)(m-k_2))$
        for some $k_1, k_2 \ge 1$ such that $k_1, k_2, m - k_1, m - k_2$ are pairwise coprime,
        \item $n = (q^d-1) / (q-1)$, $q$ odd, $d = d_1 + d_2$, $\gcd(d_1, d_2) = 1$,
        \[
            \tau = \br{\frac{q^{d_1} - 1}{q-1}, \frac{q^{d_2} - 1}{q - 1}, \frac{(q^{d_1} - 1)(q^{d_2} - 1)}{2(q-1)}, \frac{(q^{d_1} - 1)(q^{d_2} - 1)}{2(q-1)}}.
        \]
    \end{enumerate}
\end{lemma}

\section{Some exact lower bounds}

Let us now use these lemmas to obtain lower bounds for normal covering numbers.
We begin with the cases in which we can give exact bounds.

\begin{proposition}
    \label{prop:S_2p-lower}
    Assume $n > 36$ has the form $2^k$ or $2p$ ($p$ prime). Then
    \[
        \gamma(S_n) = \floor{n/4} + 1.
    \]
\end{proposition}
\begin{proof}
    The upper bound was proved in \Cref{prop:basic-upper-bounds}, so it suffices to prove the lower bound.

    Let $\cC$ be a normal covering consisting of maximal subgroups.
    Let $K$ be the set of integers $k \in [1, n/2-1]$ such that $\cC$ contains a conjugate of $S_k \times S_{n-k}$.
    We claim that $|K| \ge \floor{n/4} - 1$.
    If $K$ contains every odd number in the range $[1, n/2-1]$ then this is clear.
    Otherwise, let $x_1$ be the smallest odd element of $[1, n/2-1] \sm K$ and let $P$ be the set of $3$-partitions of the form $(x_1, x_2, x_3)$ with $x_2$ even and $x_3 > n/2$ odd.
    Note then that $\gcd(x_1, n) = \gcd(x_3, n) = 1$, $\gcd(x_2, n) = 2$, and $\gcd(x_1, x_2, x_3) = 1$.

    We claim that no $(x_1, x_2, x_3) \in P$ is covered by a transitive subgroup.
    First suppose $(x_1, x_2, x_3)$ is caught by condition \Cref{lem:imprimitive}(2).
    Since $\gcd(x_1, n) = \gcd(x_3, n) = 1$ and $\gcd(x_2, n) = 2$, the only relevant permutation of $x_1, x_2, x_3$ is $x_1, x_3, x_2$, and the condition is $(x_1 + x_3) / 2 \mid x_1, x_3$, but this is also impossible because $x_1 < n/2 < x_3$, so $(x_1 + x_3)/2 > x_1$.
    Now suppose $(x_1, x_2, x_3)$ is caught by \Cref{lem:primitive}(5) (the cases (3, 4) are irrelevant, since $n$ is even).
    We must have $n = (q^d-1) / (q-1)$ for $q$ an odd prime power, even $d = d_1 + d_2 \ge 2$, where $\gcd(d_1, d_2) = 1$.
    Since $x_3 > n/2$ we must have $x_1 = (q^{d_1} - 1) / (q-1)$, $x_2 = (q^{d_2} - 1) / (q-1)$, and $x_3 = (q^{d_1} - 1)(q^{d_2} - 1) / (q-1)$. But then $x_1$ and $x_2$ are odd and $x_3$ is even, contrary to construction.

    It follows that, for every partition $(x_1, x_2, x_3) \in P$, either the even number $x_2 < n/2$ or the odd number $n - x_3 > x_1$ must be in $K$, and moreover $K$ contains all odd numbers in $[1, x_1-2]$, so
    \[
        |K| \ge |P| + (x_1 - 1) / 2.
    \]
    Each partition $(x_1, x_2, x_3) \in P$ is determined uniquely by $x_2 < n/2 - x_1$, so
    \[
        |P| = \floor{(n/2 - x_1 - 1)/2} = \floor{n/4} - (x_1 + 1)/2.
    \]
    Therefore $|K| \ge \floor{n/4} - 1$, as claimed.

    Now in addition to $\{S_k \times S_{n-k} : k \in K\}$, $\cC$ must contain at least one transitive subgroup (containing $n$-cycles), so $|\cC| \ge |K| + 1$.
    To finish the proof it suffices to rule out the possibility that $|K| = \floor{n/4} - 1$ and $\cC$ contains a single transitive subgroup $G$.
    This subgroup $G$ would have to contain both an $n$-cycle and an element of type $(k, n-k)$ for every $k \notin K$, including $k = x_1$.
    In particular $G \ne A_n$ since $A_n$ does not contain $n$-cycles (as $n$ is even).
    Since $x_1$ is odd, \Cref{lem:imprimitive} implies $G$ must be primitive.
    By \Cref{lem:primitive}, $G$ cannot contain elements of type $(k, n-k)$ for any odd $k \in (1, n/2)$, so $x_1 = 1$ and $K$ must be precisely the set of odd numbers in the range $(1, n/2)$.
    But then $6 \notin K$, so $G$ must contain an element of type $(6, n-6)$, and since $3 \nmid n$ a power of this element has type $(3, 3, 1, \dots, 1)$, which contradicts \cite{dixon-mortimer}*{Theorem~5.4A(i)} (Jordan's bound for the minimal degree of a $2$-transitive group).
\end{proof}

To go further we need some tools from additive combinatorics, namely the Cauchy--Davenport Theorem and some extensions of it.

\begin{theorem}[Cauchy--Davenport]
    Let $A, B \subset \Zmod{p}$ be nonempty sets. Then
    \[
        |A+B| \ge \min(|A| + |B| - 1, p).
    \]
\end{theorem}

\begin{proposition}
    \label{prop:A_p-lower}
    Let $p > 36$ be a prime not of the form $(q^d  - 1) / (q - 1)$.
    Then
    \[
        \gamma(A_p)
        =
        \floor{(p+1)/3}.
    \]
\end{proposition}

\begin{remark}
    If $p$ is an odd prime of the form $(q^d - 1) / (q-1)$ then either $p$ is the Fermat prime $q+1$ or $d \ge 3$.
    Since $(q^d - 1) / (q - 1) > q^{d-1}$, it follows immediately that the number of primes up to $N$ excluded by the proposition is $O(N^{1/2})$ at most.
\end{remark}

\begin{proof}
    The upper bound was proved in \Cref{prop:basic-upper-bounds,prop:new-upper-1}, so it suffices to prove the lower bound.

    Let $\cC$ be a normal covering consisting of maximal subgroups of $A_p$.
    Let $K \subset \{1, \dots, p-1\}$ be the symmetric set of integers such that $k \in K \iff (S_k \times S_{n-k}) \cap A_n \in \cC$.
    Then by \Cref{lem:primitive} for any $3$-partition $(x, y, z)$ of $p$ we have
    \begin{equation}
        \label{eq:K-key-property}
        \begin{aligned}
            &\{x,y,z\} \cap K = \emptyset \\
            &\qquad\implies \exists \sigma \in S_3 : (x, y, z)^\sigma = (1, (p-1)/2, (p-1)/2).
        \end{aligned}
    \end{equation}
    It follows that $X = \{1, \dots, p-1\} \sm K \bmod p$ is a symmetric subset of $\Zmod p$ with the property
    \begin{equation}
        \label{eq:X-nearly-sum-free}
        x,y,z \in X ~\text{and} ~ x+y+z = 0 \implies x, y, z \in \{\pm1, \pm(p-1)/2\}.
    \end{equation}
    Indeed, if $x, y, z \in \{1, \dots, p-1\}$ and $x,y,z \in X \pmod p$ and $x+y+z \equiv 0 \pmod p$ then either $x + y + z = p$ or $x + y + z = 2p$.
    In the latter case we have $(p-x) + (p-y) + (p-z) = p$, so by symmetry of $X$ we may assume $x + y + z = p$.
    Now by the key property \eqref{eq:K-key-property} of $K$ it follows that $x, y, z \in \{\pm 1, \pm (p-1)/2\}$, as claimed.

    It follows that $|X| \le (p+5)/3$. Indeed, otherwise, $|X| \ge (p+7)/3$ and by Cauchy--Davenport $|X+X| \ge (2p+11)/3$, so
    \[
        |(-X) \cap (X+X)| \ge |X| + |X+X| - p \ge 6,
    \]
    in contradiction with \eqref{eq:X-nearly-sum-free}.
    Therefore $|K| = p - 1 - |X| \ge (2p - 8)/3$, so $\cC$ contains $|K|/2 \ge (p-4)/3$ intransitive subgroups.
    Additionally $\cC$ must contain at least one transitive subgroup (containing $p$-cycles), so $|\cC| \ge (p-1)/3$.
    By integrality it follows that $|\cC| \ge \floor{(p+1)/3}$.
\end{proof}

\section{Some approximate lower bounds}

Next we move to some approximate lower bounds. In these cases we give bounds that match the upper bounds up to a small error term.

First, even if $n$ is not prime we can give a lower bound for $\gamma(A_n)$ close to $n/3$ provided that the smallest prime factor of $n$ is large.
For this we need a result of Green and Ruzsa~\cite{green-ruzsa}*{Section~6} related to the Pollard and Kneser extensions of Cauchy--Davenport (see also \cite{hamidoune-serra}).

\begin{theorem}[Green--Ruzsa~\cite{green-ruzsa}*{Corollary~6.2}]
    Let $n \ge 1$ be a nonnegative integer, let $p_1$ be the smallest prime factor of $n$.
    For $A, B \subset \Zmod{n}$ and $K = \delta^2 n \ge 1$, let $S_K(A, B)$ denote the set of elements $x \in \Zmod{n}$ having at least $K$ representations as $a+b$ with $(a, b) \in A \times B$.
    Assume $|A|, |B| \ge \sqrt{K n}$.
    Then
    \[
        |S_K(A, B)| \ge \min(n, |A| + |B| - n/p_1) - 3 \delta n.
    \]
\end{theorem}

\begin{corollary}
    \label{cor:green--ruzsa}
    Let $X \subset \Zmod n$ and assume there are at most $\eps n^2$ additive triples $x, y, x+y \in X$. Then
    \[
        |X| \le n/3 + O(n / p_1 + \eps^{1/3}n).
    \]
\end{corollary}
\begin{proof}
    By the theorem with $A = B = X$ and $\delta = \eps^{1/3}$, one of the following alternatives holds:
    \begin{enumerate}
        \item $|X| \le \delta n$,
        \item $|S_K(X, X)| \ge 2|X| - n / p_1 - 3 \delta n$, or
        \item $|S_K(X, X)| \ge n - 3 \delta n$.
    \end{enumerate}
    We are obviously done in case (1). Suppose (2) holds. Then
    \[
        |X \cap S_K(X, X)| \ge |X| + |S_K(X, X)| - n \ge 3|X| - n - n / p_1 - 3 \delta n.
    \]
    On the other hand $K |X \cap S_K(X, X)| \le \eps n^2$. It follows that
    \[
        3 |X| \le n + n / p_1 + 4 \eps^{1/3} n.
    \]
    Case (3) is similar.
\end{proof}

Let $\tau(n)$ denote the number of divisors of $n$.
The well-known \emph{divisor bound} asserts that
\begin{equation}
    \label{eq:divisor-bound}
    \tau(n) \le n^{o(1)}.
\end{equation}
In fact $\tau(n) \le n^{O(1/\log\log n)}$.
This follows from the fundamental theorem of arithmetic and elementary analysis.
See for example \cite{tao-blog}.

\begin{proposition}
    \label{prop:A_n-large-p1-lower}
    Suppose $n$ is an odd integer with smallest prime factor $p_1$. Then
    \[
        \gamma(A_n) \ge n/3 - O(n / p_1^{1/3} + n^{2/3 + o(1)}).
    \]
\end{proposition}
\begin{proof}
    Let $\cC$ be a normal covering consisting of maximal subgroups of $A_n$.
    Let $K$ be the symmetric subset of $\{1, \dots, n-1\}$ such that $k \in K \iff (S_k \times S_{n-k}) \cap A_n \in \cC$.
    Let $X = \{1, \dots, n-1\} \sm K \bmod n$.
    Then $|K| = n-1-|X|$ and $\cC$ contains $|K|/2 = (n-1 - |X|)/2$ intransitive subgroups.
    By the previous corollary it therefore suffices to bound the number of additive triples $x, y, x+y$ in $X$.

    Suppose $x, y, x+y \in X$. By symmetry of $X$ may assume $x$ and $y$ are integers in the range $\{1, \dots, n-1\}$ such that $x + y < n$. Let $z = n - x - y$. Then $(x, y, z)$ is a $3$-partition not covered by $\{S_k \times S_{n-k} : k \in K\}$, so $(x, y, z)$ must be covered by a transitive subgroup. By \Cref{lem:imprimitive,lem:primitive}, one of the following alternatives holds.
    \begin{enumerate}
        \item $\gcd(x, y, z) > 1$,
        \item $(x+y) / \gcd(z, n) \mid x, y$ (or a permutation),
        \item $\min(x, y, z) = 1$,
        \item $(x, y, z)$ has the form in \Cref{lem:primitive}(5).
    \end{enumerate}
    We consider each case in turn.

    (1) The number of triples $(x, y, z)$ with $x+y+z = n$ and $\gcd(x, y, z) = \gcd(x, y, n) > 1$ is bounded by
    \[
        \sum_{p \mid n} (n / p)^2 \le \sum_{k \ge p_1} \frac{n^2}{k(k-1)} = \frac{n^2}{p_1-1}.
    \]

    (2) By symmetry it is enough to consider the case $(x + y) / \gcd(z, n) \mid x, y$. Let $d = \gcd(z, n)$ and $e = (x+y) / d$. Then $e \mid x, y$, so $x = ex'$ and $y = ey'$ for some integers $x', y'$ such that $x' + y' = d$. The triple $(x, y, z)$ is determined by the choice of divisor $d \mid n$ and positive integers $x' < d$ and $e < n/d$.
    It follows that the triples is less than
    \[
        \sum_{d \mid n} d \cdot (n/d) \le n \tau(n) \le n^{1+o(1)}.
    \]

    (3) The number of triples with $x+y+z=n$ and $\min(x,y,z)=1$ is obviously at most $3n$.

    (4) If $n = (q^d-1) / (q-1)$ then $d \le \log_2(n+1)$, and so there are at most $\log_2(n+1)$ pairs $(q, d)$ to consider. For each of these there are at most $d-1$ choices for $(d_1, d_2)$. Therefore the number of triples to consider is $O(\log n)^2$.

    In total we find that there are $O(n^2 / p_1 + n^{1+o(1)})$ additive triples in $X$, so \Cref{cor:green--ruzsa} implies
    \[
        |X| \le n/3 + O(p_1^{-1/3} + n^{-1/3 + o(1)}) n.
    \]
    This completes the proof.
\end{proof}

Next we turn to the case of $G = S_n$ with $n$ even or $G = A_n$ with $n$ odd.
We will prove a lower bound approximately matching the bound in \Cref{prop:BPS-upper-bound} in the case in which $6 \mid n$.

The following lemma solves the essential additive-combinatorial problem that we will reduce to.
Let us say $X \subset \Zmod n$ is \emph{coprime-sum-free} if there does not exist $x, y, x+y \in X$ with $\gcd(x, y, n) = 1$.

\begin{proposition}
    \label{prop:coprime-sum-free-2/3}
    If $X \subset \Zmod n$ is coprime-sum-free then $|X| \le (2/3)n$.
    In fact the same bound holds provided only that $X$ is free of additive triples of the form $x, x+1, 2x+1$ when $n$ is odd
    or of the form $x, x+h, 2x+h$ with $x$ odd and $h \in \{2, 4, 8\}$ when $n$ is even.
\end{proposition}
\begin{proof}
    First suppose $n$ is odd and $X$ is free of triples of the form $(x,x+1,2x+1)$. Then the indicator function $f$ of $X$ satisfies the inequality
    \[
        f(x) + f(x+1) + f(2x+1) \le 2.
    \]
    Taking the sum over $x \in \Zmod n$ gives $3|X| \le 2n$, as claimed.
    
    Now assume $n$ is even and $X$ is free of triples of the form $(x, x+h, 2x+h)$ with $x$ odd and $h \in \{2,4,8\}$.
    Let $f$ denote the indicator function of $X$ and concede that
    \begin{equation}
        \label{eq:key-f-inequality}
        f(x-h) + f(x) + f(x+h) + f(2x-h) + f(2x) + f(2x+h) \le 4
    \end{equation}
    whenever $h \in \{2, 4\}$ and $x$ is odd.
    Indeed, if $X$ contains all three of $x-h, x, x+h$ then $X$ can contain none of $2x-h, 2x, 2x+h$ since $2x-h = (x-h) + x$, $2x = (x-h) + (x+h)$, and $2x+h = x + (x+h)$,
    while if $X$ contains two of $x-h, x, x+h$ then $X$ must be missing at least one of $2x-h, 2x, 2x + h$.

    If $n \equiv 2 \pmod 4$ then by setting $h=2$ and summing \eqref{eq:key-f-inequality} over all odd $x$ we get
    \[
        3|X| = 3\sum_{\text{odd}} f + 3\sum_{\text{even}} f \le 2n.
    \]
    Finally if $n \equiv 0 \pmod 4$ then by setting $h=2,4$ respectively and summing \eqref{eq:key-f-inequality} over all odd $x$ we get
    \begin{align*}
        &3\sum_{\text{odd}} f + 2 \sum_{2 \bmod 4} f +  4 \sum_{0 \bmod 4} f&\le 2n, \\
        &3\sum_{\text{odd}} f + 6 \sum_{2 \bmod 4} f &\le 2n.
    \end{align*}
    Taking $3$ times the first inequality plus the second gives $12 |X| \le 8n$, which again gives the claim.
\end{proof}

\begin{remark}
    \label{remark:problem-1}
    If $n$ is divisible by $6$, the example
    \[
        X = \{x \in \Zmod n : 2\mid x ~ \text{or} ~ 3 \mid x\}
    \]
    shows that the constant $2/3$ is sharp (cf.~the definition of $\cC_1$ in the proof of \Cref{prop:BPS-upper-bound}).
    For other $n$, finding the cardinality of the largest coprime-sum-free subset of $\Zmod n$ is an interesting additive-combinatorial problem that we have not been able to solve in general.
    Generalizing the above example, if $\omega(n) \ge 2$ and $p_1$ and $p_2$ are the smallest prime factors of $n$ then
    \[
        X = \{x \in \Zmod n : p_1 \mid x ~ \text{or} ~ p_2 \mid x\}
    \]
    is a coprime-sum-free set of density $1/p_1 + 1/p_2 - 1/(p_1p_2)$.
    On the other hand there is always an ordinary sum-free set of density $1/3 - o(1)$ (as $n\to\infty$), which is greater if $p_1 > 3$.
    It may be that if $\gcd(n, 6) = 1$ then the largest density of a coprime-sum-free set is about the same as that of the largest ordinary sum-free set.
\end{remark}

\begin{proposition}
    \label{prop:n/6-lower-bound}
    Assume $G = S_n$ and $n$ is even or $G = A_n$ and $n$ is odd. Then
    \[
        \gamma(G) \ge n / 6 - n^{o(1)}.
    \]
\end{proposition}
\begin{proof}
    Let $\cC$ be a normal covering consisting of maximal subgroups of $G$.
    Let $K$ be the symmetric subset of $\{1, \dots, n-1\}$ such that $k \in K \iff (S_k \times S_{n-k}) \cap G \in \cC$
    Let $X = \{1, \dots, n-1\} \sm K \bmod n$.
    As in previous arguments we have $|\cC| \ge (n-1-|X|)/2$, so it suffices to bound $|X|$.
    We will argue that $|X|$ is bounded by roughly $(2/3)n$ using \Cref{prop:coprime-sum-free-2/3} and the observation that $X$ is almost coprime-sum-free, as before using \Cref{lem:imprimitive,lem:primitive} to make this precise.

    By \Cref{prop:coprime-sum-free-2/3}, it would be enough to remove a small symmetric set $T$ from $X$ so that $X \sm T$ is free of triples of the form $x, x+h, 2x+h$ where $h \in \{1, 2, 4\}$ and $\gcd(x, h, n) = 1$.
    Suppose $x, x+h, 2x+h \in X$ and $\gcd(x, h, n) = 1$.
    By including the intervals $[-2, 2]$ and $[n/2-2, n/2+2]$ in $T$ and using symmetry of $X$ and $T$ we may assume $2 < x, x+h < n/2-2$, so $2x+h < n$.
    Now \Cref{lem:imprimitive,lem:primitive} apply to the triple $(x, x+h, n - 2x - h)$, which must be covered by some transitive subgroup in $\cC$. Note that $\gcd(x, x+h, n-2x-h) = \gcd(x, h, n) = 1$.
    Since $1 \in T$, the cases \Cref{lem:primitive}(3, 4) are already eliminated.
    We may eliminate \Cref{lem:primitive}(5) by including in $T$ all $O(\log n)^2$ points of the form $\pm (q^{d_1}-1) / (q-1)$ with $n = (q^d-1)/(q-1)$.
    Now consider \Cref{lem:imprimitive}(2). There are three conditions to consider:
    \begin{enumerate}
        \item $(2x+h) / \gcd(2x+h, n) \mid x, x+h$,
        \item $(n-x-h) / \gcd(n-x-h, n) \mid x, n-2x-h$,
        \item $(n-x) / \gcd(n-x, n) \mid x+h, n-2x-h$.
    \end{enumerate}
    We consider each case in turn for fixed $h \in \{1, 2, 4\}$.

    (1) We have $(2x+h) / \gcd(2x+h, n) \mid h$, so $2x+h = de$ where $d \mid n$ and $e \mid h$.
    The number of solutions is therefore bounded by $2\tau(n) \tau(h)$.

    (2) In this case $(n - x - h) / \gcd(n - x- h, n) \mid n - h$, so $n - x - h = de$ where $d \mid n$ and $e \mid n - h$. The number of solutions is therefore bounded by $\tau(n) \tau(n-h)$.

    (3) Similarly, $(n - x) / \gcd(n - x, n) \mid n + h$, so $n - x = de$ for some $d \mid n$ and $e \mid n+h$.
    The number of solutions is therefore bounded by $\tau(n) \tau(n+h)$.

    Applying the divisor bound \eqref{eq:divisor-bound} and summing over $h = 1, 2, 4$, it follows that the total number of triples is only $n^{o(1)}$.
    Pick a point $x$ from each such triple and include $\pm x$ in $T$.
    Then we have obtained the required set $T$, and $|T| \le n^{o(1)}$, so \Cref{prop:coprime-sum-free-2/3} implies $|X \sm T| \le (2/3) n$,
    so $|X| \le (2/3)n + n^{o(1)}$.
\end{proof}

Finally we give a lower bound in the case $G = S_n$ with $n$ odd or $G = A_n$ with $n$ even that is approximately within a factor of $2$ of the main term in the upper bound \Cref{prop:new-upper-2}.

Again we start by isolating the essential additive-combinatorial problem.
For elements $x, y, z$ in an abelian group define
\[
    \cube(x, y, z) = \{x, y, z, x+y, x+z, y+z, x+y+z\}.
\]
Call $X \subset \{1, \dots, n-1\}$ \emph{coprime-cube-free} if there does not exist $x, y, z$ with $\gcd(x, y, z, n) = 1$ and $\cube(x, y, z) \subset X$.
The following proposition gives a bound for the size of any symmetric coprime-cube-free subset of $\{1, \dots, n-1\}$.

\begin{proposition}
    \label{prop:coprime-cube-free}
    If $X \subset \{1, \dots, n-1\}$ is symmetric and coprime-cube-free then
    \[
        |X| \le \tfrac{5}{6} n + O(1).
    \]
    In fact the same bound holds provided only that $X$ is free of degenerate cubes of the form
    \[
        \cube(x, x, x+h) = \{x, x+h, 2x, 2x+h, 3x+h\} \qquad (h = \pm 1).
    \]
\end{proposition}
\begin{proof}
    Let $f$ denote the indicator function of $X$. Then for all $x$ and $h = \pm1$ we have identically
    \begin{equation}
        \label{eq:f4}
        f(x) + f(x+h) + f(2x) + f(2x+h) + f(3x+h) \le 4.
    \end{equation}
    Consider the sum over all $x \in [0, n/3]$ and $h = \pm 1$.
    With $O(1)$ error (the edge effect of values of $x$ near the ends of the interval),
    the linear forms $x, x+h$ run each twice over the interval $[0, n/3]$,
    the linear form $2x$ runs twice over the even integers in the interval $[0, 2n/3]$,
    the linear from $2x+h$ runs twice over the odd integers in the interval $[0, 2n/3]$, and lastly
    the linear form $3x+h$ runs once over the ingers in the interval $[0,n]$ not divisible by $3$.
    Therefore we obtain
    \[
        4\sum_{[0, n/3]} f + 2\sum_{[0, 2n/3]} f + \sum_{[0,n] \sm 3\Z} f \le 8n/3 + O(1),
    \]
    where the error term takes care of the edge effects. By symmetry of $f$ we have
    \[
        4\sum_{[0,n/3]} f = 2\sum_{[0,n/3] \cup [2n/3,n]} f = 2\sum_{[0,n]} f - 2\sum_{[n/3, 2n/3]} f + O(1)
    \]
    and
    \[
        2\sum_{[0,2n/3]} = \sum_{[0,n]} f + \sum_{[n/3,2n/3]} f + O(1),
    \]
    so it follows that
    \[
        4 \sum_{[0,n]} f - \sum_{[n/3,2n/3]} f - \sum_{[0,n] \cap 3\Z} f \le 8n/3 + O(1).
    \]
    Since the sums of $f$ over $[n/3, 2n/3]$ and $[0,n]\cap 3\Z$ are at most $n/3 + 1$, it follows that
    \[
        4|X| = 4 \sum_{[0,n]} f \le 10n/3 + O(1),
    \]
    so $|X| \le 5n/6 + O(1)$, as claimed.
\end{proof}

\begin{remark}
    \label{remark:problem-2}
    The constant $5/6$ is almost certainly not sharp.
    We believe that $2/3$ is the optimal constant, which is attained by the example
    \[
        X = \{x \in \{1, \dots, n-1\} : 3 \nmid x\}
    \]
    whenever $n$ is divisible by $3$ (cf.~the definition of $\cC_1$ in \Cref{prop:new-upper-2}),
    and a stronger bound probably holds if $3 \nmid n$.

    There is a conjecture attributed to the first author and Pohoata that the largest (ordinary, not necessarily symmetric) cube-free subset of $\{1, \dots, n-1\}$ has density at most $2/3 + o(1)$, which is attained by both the above example and $[n/3, n]$ (see Meng~\cite{meng}).
    We can prove this in the ordinary symmetric case.
\end{remark}

The following proposition is not as natural, but it turns out to be needed in what follows.

\begin{proposition}
    \label{prop:coprime-cube-free-2}
    Assume $n$ is even. If $X \subset \{1, \dots, n-1\}$ is symmetric and free of cubes of the form $\cube(x, x, x+h)$ with $h = \pm1$ and $x$ even then
    \[
        |X| \le \tfrac{8}{9}n + O(1).
    \]
\end{proposition}
\begin{proof}
    The proof is similar. Summing \eqref{eq:f4} over all even $x \in [0, n/3]$ and $h = \pm1$ gives
    \[
        2 \sum_{[0,n/3]} f + 2 \sum_{[0,2n/3] \cap 4\Z} f + \sum_{[0,2n/3] \cap (2\Z+1)} f + \sum_{[0,n] \cap (6\Z \pm 1)} f \le 4n/3 + O(1).
    \]
    Adding the three trivial inequalities
    \begin{align*}
    2\sum_{[0,2n/3] \cap (4\Z+2)} f \le n/3, &&
    \sum_{[0,2n/3] \cap (2\Z+1)} f \le n/3, &&
    \sum_{[0,n] \sm (6\Z\pm 1)} f \le 2n/3,
    \end{align*}
    we obtain
    \[
        2 \sum_{[0,n/3]} f + 2 \sum_{[0,2n/3]} f + \sum_{[0,n]} f \le 8n/3 + O(1).
    \]
    By symmetry of $f$, the left-hand side is equal to $3 \sum_{[0,n]} f = 3 |X|$.
    The claimed bound follows.
\end{proof}

\begin{proposition}
    \label{prop:n/12-lower-bound}
    Let $G = S_n$ with $n$ odd or $G = A_n$ with $n$ even. Then
    \[
        \gamma(G) \ge \begin{cases}
            n / 12 - n^{o(1)} &\text{if}~ G = S_n,~n~\text{odd},\\
            n / 18 - n^{o(1)} &\text{if}~ G = A_n,~n~\text{even}.
        \end{cases}
    \]
\end{proposition}
\begin{proof}
    Let $\cC$ be a normal covering consisting of maximal subgroups of $G$ and as usual let $K \subset \{1, \dots, n-1\}$ be the set of integers $k$ such that $\cC$ contains a conjugate of $(S_k \times S_{n-k}) \cap G$. Let $X = \{1, \dots, n-1\} \sm K$.
    We claim there is a small set $T \subset \{1, \dots, n-1\}$ so that $X \sm (T \cup (n - T))$ is free of cubes of the form $\cube(x,x,x+h)$ with $h = \pm1$ and $\gcd(n, x+1, 2) = 1$.

    Suppose $\cube(x, x, x+h) \subset X$. Then the partition $(x, x, x+h, n - 3x-h)$ is a coprime $4$-partition of $n$ not covered by $\{(S_k \times S_{n-k}) \cap G : k \in K\}$.
    Since permutations with $4$ cycles are contained in $A_n$ when $n$ is even and $S_n \sm A_n$ when $n$ is odd, $(x, x, x+h, n-3x-h)$ must be covered by a transitive subgroup in $\cC$ smaller than $A_n$.

    First consider imprimitive subgroups. Applying \Cref{lem:imprimitive}, one of the following cases holds for some permutation $x_1, x_2, x_3, x_4$ of the partition $x, x, x+h, n-3x-h$.
    \begin{enumerate}[(a)]
        \item ($\Pi$ has type $(3, 1)$) We have $b = \gcd(x_4, n) > 1$ and
        \[
            (x_1 + x_2 + x_3) / b \mid x_1, x_2, x_3.
        \]
        \begin{enumerate}[1.]
            \item ($x_4 = n-3x-h$) $(3x-h) / \gcd(3x-h, n) \mid x, x+h$. Since $\gcd(x, x+h) = 1$, in this case we must have $3x - h \mid n$. The number of possibilities is at most $\tau(n)$.
            \item ($x_4 = x+h$) $(n - x - h) / \gcd(x + h, n) \mid x, n - 3x - h$. In this case $n - x - h = bc$ for some $b \mid n$ and some $c \mid n - h$. The number of possibilities is at most $\tau(n) \tau(n-h)$.
            \item ($x_4 = x$) $(n - x) / \gcd(x, n) \mid x+h, n - 3x - h$. In this case $n - x = bc$ for some $b \mid n$ and some $c \mid n + 2h$. The number of possibilities is at most $\tau(n) \tau(n + 2h)$.
        \end{enumerate}
        \item ($\Pi$ has type $(2, 1, 1)$) We have $b = \gcd(x_3, x_4, n) > 1$ and
        \[
            (x_1 + x_2) / b \mid x_1, x_2.
        \]
        Note that the hypothesis $\gcd(n, x+1, 2) = 1$ eliminates the important case $x_1 = x_2 = x$. In fact $\gcd(x_3, x_4, n) > 1$ is only possible if $x_3 = x_4 = x$. In this case we must have $(n - 2x) / \gcd(x, n) \mid x+h, n - 3x - h$, which implies that $n - 2x = bc$ for some $b \mid n$ and some $c \mid n + 2h$, so the number of possibilities is at most $\tau(n) \tau(n + 2h)$.
        \item  ($\Pi$ has type $(2, 2)$) We have $b = \gcd(x_1 + x_2, n) = \gcd(x_3 + x_4, n) > 1$ and
        \[
            (x_1 + x_2) / b \mid x_1, x_2 \qquad\text{and}\qquad (x_3 + x_4) / b \mid x_3, x_4.
        \]
        There are essentially just two cases, depending on whether the two copies of $x$ appear in the same part.
        \begin{enumerate}[1.]
            \item ($x_1 = x_2 = x$) $2x / \gcd(2x, n) \mid x$ and $(n-2x) / \gcd(2x, n) \mid x+h, n-3x-h$. In this case ($n$ is even and) $n - 2x = bc$ for some $b \mid n$ and some $c \mid n + 2h$. The number of possibilities is at most $\tau(n) \tau(n + 2h)$.
            \item ($x_1 = x, x_2 = x+h$) $(2x + h) / \gcd(2x + h, n) \mid x, x+h$ and $(n-2x-h) / \gcd(2x + h, n) \mid x, n - 3x-h$. In this case $2x+h \mid n$, so the number of possibilities is at most $\tau(n)$.
        \end{enumerate}
    \end{enumerate}
    All together the number of cases is at most
    \[
        \sum_{h = \pm1} (2\tau(n) + \tau(n) \tau(n-h) + 3\tau(n)\tau(n+2h)) \le n^{o(1)}
    \]
    by the divisor bound \eqref{eq:divisor-bound}.
    For each such cube $\cube(x, x, x+h)$, include $x$ in $T$.

    Next consider primitive subgroups.
    By including $1$ in $T$ we eliminate the cases \Cref{lem:primitive}(6--8),
    but the other two cases require further argument.

    Consider \Cref{lem:primitive}(9). If $(x, x, x+h, n-3x-h)$ has the form $\tau = (k_1k_2, k_1(m-k_2), (m-k_1)k_2, (m-k_1)(m-k_2))$ we may assume (possibly swapping $k_1$ and $k_2$, or $k_1$ and $m-k_1$, or $k_2$ and $m-k_2$)
    that $k_1k_2 = x$.
    If $k_1(m-k_2) = x$ then we must have $k_2 = m/2$, in contradiction with $\gcd(k_2, m-k_2) = 1$, and similarly if $(m - k_1)k_2 = x$, so we must have $(m-k_1)(m-k_2) = x$ and say $k_1(m-k_2) = x + h$.
    Since $\gcd(x, x+h) = 1$ this implies that $k_1 = 1$ and $k_2 = x = m - x - h$, but then $(m-k_1)(m-k_2) = (m-1)(x+h) > x$.
    Therefore this case does not arise.
    (Alternatively, by the solution to Erd\H{o}s's multiplication table problem, the set of all products $k_1k_2$ with $k_1, k_2 \le m$ has cardinality $O(m^2 / (\log m)^c) = O(n / (\log n)^c)$ for some $c > 0$, so we may just include all such elements in $T$ at the cost of a somewhat poorer error bound.)

    Finally consider \Cref{lem:primitive}(10).
    Suppose the quadruple $(x, x, x+h, n-3x-h)$ is equal to
    \[
        \tau = \br{\frac{q^{d_1} - 1}{q-1}, \frac{q^{d_2} - 1}{q - 1}, \frac{(q^{d_1} - 1)(q^{d_2} - 1)}{2(q-1)}, \frac{(q^{d_1} - 1)(q^{d_2} - 1)}{2(q-1)}}
    \]
    up to permutation of the entries.
    As in the proof of \Cref{prop:n/6-lower-bound}, we can eliminate this case by including in $T$ all $O(\log n)^2$ elements of the form $(q^{d_1} - 1) / (q-1)$ with $d_1 \le d$ and $n = (q^d-1) / (q-1)$.
    (Alternatively, one may check directly that this case does not arise for $n \ge 6$.)

    Thus we have a set $T \subset \{1, \dots, n-1\}$ of size $|T| \le n^{o(1)}$ such that $Y = X \sm (T \cup (n-T))$ is free of cubes of the $\cube(x, x, x+h)$ with $h = \pm1$ and $\gcd(n, x+1, 2) = 1$.
    By \Cref{prop:coprime-cube-free,prop:coprime-cube-free-2}, we have
    \[
        |Y| \le \begin{cases}
            (5/6) n + O(1) &\text{if $n$ is odd}, \\
            (8/9) n + O(1) &\text{if $n$ is even},
        \end{cases}
    \]
    and obviously $|X| \le |Y| + 2|T| \le |Y| + n^{o(1)}$.
    Thus $|K| = n-1 - |X| \ge n - |Y| - n^{o(1)}$, and $\cC$ contains at least $|K|/2$ different intransitive subgroups.
    This completes the proof.
\end{proof}

\section{Limits}

\begin{theorem}[\Cref{thm:limits} restated]
    \label{thm:limits-restated}
    We have the following asymptotic estimates for $\gamma(S_n)$ and $\gamma(A_n)$:
    \begin{align*}
        &\limsup_{n~\textup{even}} \gamma(S_n) / n = 1/4,
        &&\limsup_{n~\textup{even}} \gamma(A_n) / n = 1/4, \\
        &\limsup_{n~\textup{odd}} \gamma(S_n) / n = 1/2,
        &&\limsup_{n~\textup{odd}} \gamma(A_n) / n = 1/3, \\
        &\liminf_{n~\textup{even}} \gamma(S_n) / n = 1/6,
        &&\liminf_{n~\textup{even}} \gamma(A_n) / n \in [1/18, 1/6], \\
        &\liminf_{n~\textup{odd}} \gamma(S_n) / n \in [1/12, 1/6],
        &&\liminf_{n~\textup{odd}} \gamma(A_n) / n \in [1/6, 4/15]. \\
    \end{align*}
\end{theorem}

\begin{proof}
    The eight cases naturally split into 4 groups of 2 according to the formula $\opr{sgn}(\pi) = (-1)^{n-k}$, where $k$ is the number of cycles of $\pi$.
    Recall that to bound a limsup above we need to consider the whole sequence, while to bound a limsup below we can pass to any subsequence we like, and exactly the reverse for liminfs.

    \emph{Cases $\limsup_{n~\textup{even}} \gamma(A_n) / n$ and $\limsup_{n~\textup{odd}} \gamma(S_n) / n$:}
    See \Cref{prop:basic-upper-bounds,remark:known-exact-things}. These results are already in \cite{BP2011}.
    The limsups are attained along the sequences $2p$ and $p$ respectively.

    \emph{Cases $\limsup_{n~\textup{even}} \gamma(S_n) / n$ and $\limsup_{n~\textup{odd}} \gamma(A_n) / n$:}
    See \Cref{prop:basic-upper-bounds} for the upper bounds and \Cref{prop:S_2p-lower,prop:A_p-lower} for the lower bounds.
    The limsups are attained again along the sequences $2p$ and $p$ respectively.
    According to \Cref{prop:A_n-large-p1-lower}, in the latter case the limit is also attained along any sequence in which the smallest prime factor of $n$ tends to infinity.

    \emph{Cases $\liminf_{n~\textup{even}} \gamma(S_n) / n$ and $\liminf_{n~\textup{odd}} \gamma(A_n) / n$:}
    See \Cref{prop:BPS-upper-bound} for the upper bounds and \Cref{prop:n/6-lower-bound} for the lower bounds.
    In the former case, the liminf is attained along the sequence $n = 6m$.

    \emph{Cases $\liminf_{n~\textup{even}} \gamma(A_n) / n$ and $\liminf_{n~\textup{odd}} \gamma(S_n) / n$:}
    The lower bounds are immediate from \Cref{prop:n/12-lower-bound}.
    To establish the upper bounds we may restrict to a subsequence in $\phi(n) / n \to 0$.
    For example if $p_1, \dots, p_k$ are the first $k$ odd primes then we may take $n = 2p_1\cdots p_k$ or $n = p_1 \cdots p_k$.
    For these sequences, \Cref{prop:new-upper-2} implies $\liminf \gamma(G) / n \le 1/6$.
\end{proof}

\bibliography{refs}

\end{document}